\documentclass[11pt,notitlepage,twoside,a4paper]{amsart}
 \usepackage{amsfonts}
 \usepackage{mathrsfs}
\usepackage{amsmath,amssymb,enumerate}

\usepackage{epsfig,fancyhdr,color}

\usepackage{amssymb}
\usepackage{amsmath,amsthm}
\usepackage{latexsym}
\usepackage{amscd}
\usepackage{psfrag}
\usepackage{graphicx}
\usepackage[latin1]{inputenc}
\usepackage[all]{xy}

\newtheorem{theorem}{\rm\bf Theorem}[section]

\newtheorem{lemma}[theorem]{\rm\bf Lemma}

\newtheorem*{theorem 1}{\rm\bf Proposition 1}
\newtheorem*{theorem 2}{\rm\bf Proposition 2}
\newtheorem*{proposition 1}{\rm\bf Proposition 3.2}

\theoremstyle{definition}
\newtheorem{definition}[theorem]{\rm\bf Definition}

\theoremstyle{remark}
\newtheorem{remark}[theorem]{\rm\bf Remark}

\newtheorem{example}[theorem]{\rm\bf Example}

\def\interieur#1{\mathord{\mathop{\kern 0pt #1}\limits^\circ}}

\definecolor{NoteColor}{rgb}{1,0,0}

\title[Spherical arc-length as a global conformal parameter]{Spherical arc-length as a global conformal parameter for analytic curves in the Riemann sphere}
\author{Paul Gauthier, Vassili Nestoridis and Athanase Papadopoulos}
\address{Paul Gauthier, D\'epartement de math\'ematiques et de statistique, Universit\'e de Montr\'eal,  Pavillon Andr\'e-Aisenstadt, 2920, chemin de la Tour, Montr\'eal (Qu\'ebec)  H3T 1J4 Canada, 
 email: 
 {\rm gauthier@dms.umontreal.ca}}
\address{Vassili Nestoridis, Department of Mathematics, University of Athens, Department of mathematics, 15784 Panepistemioupolis, Athens, Greece,
 email: 
 {\rm vnestor@math.uoa.gr}}
 \address{Athanase Papadopoulos,
Institut de Recherche Math\'ematique Avanc\'ee, Universit\'e de Strasbourg et CNRS,
7 rue Ren\'e Descartes,
67084 Strasbourg Cedex France,
 email: {\rm papadop@math.unistra.fr} }

\date{\today}

\thanks{This research was supported through the programme ``Research in Pairs"  of the Mathematisches Forschungsinstitut Oberwolfach in 2016, as well as by NSERC (Canada) }

\begin{document}

\maketitle


\begin{abstract} We prove that for every analytic curve in the complex plane $\mathbb{C}$, Euclidean and spherical arc-lengths are global conformal parameters. We also prove that for any analytic curve in the hyperbolic plane,  hyperbolic arc-length is also a global parameter. We generalize some of these results to the case of analytic curves in $\mathbb{R}^n$ and  $\mathbb{C}^n$ and we discuss the situation of curves in the Riemann sphere $\mathbb{C}\cup\{\infty\}$. 
\end{abstract}

\medskip

\medskip

%
%
%
%
%
                 
\section{Introduction}

The subject of this paper is the interaction between real analytic curves and complex geometry. The results are in the direction of those obtained in the papers \cite{NP1, NP2}.

In the papers \cite{Bolkas} and \cite{Bolkas2}, the authors obtained results on the parametrization of an analytic Jordan curve induced by the Riemann mapping theorem from the open unit disc $\mathbb{D}$ onto the interior of that curve. A natural question was addressed, namely, whether the same results are valid when the arc-length parametrization of the curve is used. This led two of the authors of the present paper to prove that arc-length is a global conformal parameter for any analytic curve \cite{NP1, NP2}. Thus, nothing changes in the results of  \cite{Bolkas, Bolkas2} if we use the arc-length parametrization of the curve considered.

On the other hand, it may happen that with a particular conformal parametrization of an analytic curve, it is not clear whether one can further extend this curve analytically to a domain larger than its domain of definition.  It was shown in \cite{NP1} and \cite{NP2} that in order to be sure whether such an extension is possible or not, one may simply use the arc-length parametrization, and examine whether the extendability for this parametrization is possible.

In the present paper, we show that also spherical arc-length is a global conformal parameter for any analytic curve in $\mathbb{C}$. Moreover, for analytic curves in the upper half-plane model of hyperbolic geometry, we show that hyperbolic length is also a global conformal parameter. These results are proved in \S \ref{s:analytic} and are corollaries of a more general theorem involving holomorphic functions in several complex variables, where each variable is replaced by a derivative of the complexification of the function $\gamma$ defining the analytic curve or by the complex conjugate of a derivative of $\gamma$ evaluated at the complex conjugate $\overline{z}$ of the variable $z$ in $\mathbb{C}$.

In \S \ref{s:n}, we extend the result of \S \ref{s:analytic} to analytic curves in $\mathbb{R}^n$ and $\mathbb{C}^n$. Essentially, the same methods of proof work.

In \S \ref{s:ext}, we consider the notion of analytic curve in the Riemann sphere $\mathbb{C}\cup\{\infty\}$. The more general form of this definition is that locally the function $\gamma$ is the restriction of an injective meromorphic function. We show that the spherical arc-length of such a curve is a global conformal parameter. It follows that the maximal extension of such a curve parametrized by spherical arc-length is always defined on an open subinterval of $\mathbb{R}=(-\infty,+\infty)$ and cannot contain the point $+\infty$. This contrasts with what happens if we use other parametrizations, for instance the Euclidean arc-length. We note that these parametrizations by Euclidean arc-length or spherical arc-length should define strictly increasing functions, and in most cases we are led to allow the parameter to take negative values, although we call it arc-length parameter or spherical arc-length parameter (see \cite{NP2}, beginning of \S,3, where this question is discussed). 

Finally, we mention that our definitions are special cases of more general definitions of analytic curves on Riemann surfaces, but we do not insist on that in the present paper.

In this paper, by a \emph{conformal} mapping we mean a locally conformal mapping, that is, holomorphic and locally injective. (Equivalently, the map is holomorphic with nonzero derivative, \cite{Ahlfors}.)


\section{Analytic curves in $\mathbb{C}$}\label{s:analytic}
\begin{definition}
Let $I\subset \mathbb{R}$ be a nonempty interval of an arbitrary type which is not a singleton and let $\gamma:I\to \mathbb{C}$ be a continuous map. We say that $\gamma$ is an \emph{analytic curve} if $\gamma$ is differentiable on $I$ and $\gamma'(t)\not=0$ for all $t\in I$ and if for every $t_0\in I$ there exists $\delta=\delta_{t_{0}}>0$ and a power series $\sum_{n=0}^\infty a_n (t-t_0)^n$, $a_n\in \mathbb{C}$, which is convergent on $(t_0-\delta,t_0+\delta)$ and such that $\gamma(t)=\sum_{n=0}^\infty a_n(t-t_0)^n$ for all $t\in I\cap(t_0-\delta,t_0+\delta)$. 
\end{definition}

It follows from this definition that $\gamma$ is locally injective and that its maximal analytic extension is defined on an open subinterval of $\mathbb{R}$, see \cite{NP2}. It is also true that if we restrict $\gamma$ to a compact interval $I_0$ such that $\gamma_{\vert I_0}$ is injective, then $\gamma$ has an injective holomorphic extension to an open neighborhood $V\subset \mathbb{C}$ of $I_0$ (see \cite{NP1}, \cite{NP2}). In fact, this local extendability to an injective holomorphic function is equivalent to the definition of an analytic curve \cite{Ahlfors, Gamelin, MO}.

Let $\gamma:[a,b]\to \mathbb{C}$
 be an analytic curve and $N\geq 1$ a natural number. We consider the image in $\mathbb{C}^{2N}$ of $[a,b]$ by the map
\begin{equation*}
\Gamma(t)=\big(\gamma(t),\overline{\gamma(\overline{t}}),
\gamma'(t),\overline{\gamma'(\overline{t}}),
\ldots,
\gamma^{(N)}(t),\overline{\gamma^{(N)}(\overline{t}})\big).
\end{equation*}
(Note that since $t$ is real, we have $\overline{t}=t$, but the  above notation is also meant to  deal 
with the case where $t$ is complex, which we shall consider shortly.)

We denote by $J\subset \mathbb{C}^{2N} $ the image of $[a,b]$ by the map $\Gamma$.

We now consider an open subset $\Omega$ of $\mathbb{C}^{2N} $ containing $J$, and a holomorphic map 
\[F:\Omega\to \mathbb{C}\]
such that $F(J)\subset (0,+\infty)$.

We set 
\[S(t)=\int_a^t F\big(\Gamma(u)\big)du\]
and we let $[A,B]$ be the image of $[a,b]$ by $S$.
Then, $S$ is a strictly increasing and continuous function from $[a,b]$ onto $[A,B]$. Therefore it has an inverse 
\[t=S^{-1}(s),\]
and  we can consider a new parametrization of the initial curve $\gamma$ by setting 
\
\[\delta(s)=\gamma(S^{-1}(s)).\]

\begin{theorem}\label{th:1}
The curve $\delta$ is analytic.
\end{theorem}
In other words, the curve $\gamma$  remains analytic with respect to this new parametrization.
\begin{proof}
Since the map $\gamma$ is locally injective, we can restrict $\gamma$ to a smaller interval to obtain an injective map. Without loss of generality we may assume that $\gamma$ is injective on $[a,b]$. 
Therefore the curve $\gamma$ has an injective holomorphic extension on a neighborhood $V\subset\mathbb{C}$ of $[a,b]$. We again call $\gamma$ this extension.
 
$F(\Gamma(t))$ is defined on $[a,b]$. There exists a neighborhood $V_1$ of $[a,b]$ in $\mathbb{C}$ which is contained in $V$, such that the the following function has a holomorphic extension in that neighborhood:
\[G(z)= F(\gamma(z),\overline{\gamma(\overline{z})},
\gamma'(z),\overline{\gamma'(\overline{z})},
\ldots,
\gamma^{(N})(z),\overline{\gamma^{(N)}(\overline{z})}
.\]

We can choose this neighborhood $V_1$ of $[a,b]$ in $\mathbb{C}$ (the domain of $z$) to be convex, e.g. a rectangle. Thus, the function $G$ has a holomorphic primitive (which we shall also call $S$) in $V_1$, and the primitive may be chosen so that it coincides with $S$ at the point $a$ (that is, it takes the value 0 at that point). 
Thus, this function $S$ is an extension, on $V_1$ of the initial function $S$ defined on $[a,b]$.

The restriction of $S$ to $[a,b]$ is injective. The derivative of $S$ is $F(\Gamma(t))$ and is nonzero on $[a,b]$, since we have $F(J) \subset (0,+\infty)$. We then prove as in the papers \cite{NP1} and \cite{NP2} that we can take a smaller neighborhood of $[a,b]$ which is a rectangle such that the restriction of $S$ to that rectangle is injective.

This shows that the new parametrization of the initial curve $\delta$ is an analytic function.
The proof is complete.
\end{proof}
 
 Theorem \ref{th:1} also follows from the following lemma:

\begin{lemma}
Let $f$ be analytic on  $[0,b].$ Then 
$$	
	S(t) = \int_0^t f(u)du
$$
is analytic and hence extends holomorphically to a neighborhood of $[0,b].$ If $f>0,$ then this extension is locally conformal in a (possibly smaller) neighborhood of $[0,b].$
\end{lemma}
\begin{proof} 
$$
	S^\prime(t) = f(t) = \sum_{k=0}^\infty a_k(t-t_0)^k, \quad |t-t_0|<\delta_0,
$$
$$
	S(t)  = \sum_{k=0}^\infty \frac{a_k}{k+1}(t-t_0)^{k+1}, \quad |t-t_0|<\delta_0.
$$
\end{proof}
Note that if $f$ is positive, then $S'$ is also positive, so the extension is locally conformal  in a smaller neighborhood.

Now, let $\gamma:[0,b]\rightarrow \mathbb{C}$ be an analytic curve. Let $\Gamma:[0,b]\rightarrow \mathbb{C}^n$ be the analytic curve $\Gamma = (\gamma,\gamma^\prime,\ldots,\gamma^{(n)}).$ Suppose $F$ is holomorphic on a neighborhood of $J=f([0,b])$ and positive on $J.$ Define the \emph{length} (or \emph{$F$-length}) of $\gamma([0,t])$ as 
$$	
	S(t) = \int_0^t (F\circ\Gamma)(u)du.
$$
We may apply the lemma to $S,$ with $f=F\circ\Gamma.$ We obtain that length is a conformal parameter.  
Thus, we get another proof of Theorem \ref{th:1}.

We now give three applications of Theorem \ref{th:1}, in the form of examples.

\begin{example}[Euclidean arc-length]
We consider the particular case where $F$ is the function
\[F(z)=\sqrt{z_3z_4}\] for
$z=(z_1,z_1, z_3,z_4)\in \mathbb{C}^4$, with $F$ defined on 
\[\Omega=\{(a_1,z_2,z_3,z_4)\in \mathbb{C}^4 : \ \mathrm{Re}z_3>0 , \  \mathrm{Re}z_4>0.\}\]
Note that in this case we have $z_3z_4\not\in (-\infty,0]$, therefore the square root function is well defined and holomorphic, and the function $F\circ \Gamma$ is (locally) conformal.

Then, we have 
\[(F\circ \Gamma)(t)=\vert \gamma'(t)\vert=\sqrt{\gamma'(t). \overline{\gamma'(\overline{t})}}\]
and
\[G(z)=\sqrt{\gamma'(z). \overline{\gamma'(\overline{z})}}.\]
Theorem \ref{th:1} gives another proof of Theorem 3.1 of \cite{NP2} whose statement is the following:

\begin{theorem}
Arc-length is a global conformal parameter for any analytic curve in $\mathbb{C}$.
\end{theorem}
\end{example}

\begin{example}[Spherical arc-length]
The derivative of spherical arc-length with respect to the parameter $t$ of a curve is given by the following formula:
\[\frac{\vert\gamma'(t)\vert}{1+\vert\gamma(t)\vert^2}=\frac{\sqrt{\gamma'(t).\overline{\gamma'(\overline{t}})}}{1+\gamma(t)\overline{\gamma(\overline{t})}}.
\]
Thus, in this case the function $F$ of Theorem \ref{th:1} is:
\[F(z_1,z_2,z_3, z_4)=\frac{\sqrt{z_3z_4}}{1+z_1z_2},\]
defined on 
\[\{(z_1,z_2,z_3,z_4)\in\mathbb{C}^4: \mathrm{Re} z_3>0, \ \mathrm{Re} z_4>0, \  \mathrm{Re}z_1z_2>-1/2\}.\]

Theorem \ref{th:1} implies in this case the following:

\begin{theorem}
For any analytic curve in $\mathbb{C}$, spherical arc-length is a global and conformal parameter.
\end{theorem}
\end{example}

\begin{example}[Hyperbolic arc-length]
We consider an analytic curve $\gamma$ in the hyperbolic plane $\mathbb{H}^2$, and we use the upper half-space model:
\[\mathbb{H}^2=\{w=x+iy \ : \ y=\mathrm{Im}(w) >0\}.\]

The derivative of arc-length is given by the formula
\[
\frac{\sqrt{\big(\frac{dx}{dt}(\gamma(t))\big)^2+\big(\frac{dy}{dt}(\gamma(t))\big)^2}}{y(\gamma(t))}=\frac{\vert\gamma'(t)\vert}{\displaystyle \frac{\gamma(t)-\overline{\gamma(\overline{t})}}{2i}}
=\frac{\sqrt{\gamma'(t).\overline{\gamma'(\overline{t})}}}{\gamma(t)-\overline{\gamma(\overline{t})}}2i.
\]

Therefore, here we consider the function
\[F(z_1,z_2,z_3,z_4)=\frac{\sqrt{z_3z_4}}{z_1-z_2}2i\]
defined on 
\[\{(z_1,z_2,z_3,z_4)\in\mathbb{C}^4 : \mathrm{Re}z_3>0, \ \mathrm{Re}z_4>0,\ 
\mathrm{Im}\big(\frac{z_1-z_2}{2i}\big)>0\}.\]
Note that $\mathrm{Re}(z_1-z_2)<0$ is equivalent to $\mathrm{Im}\left(\frac{z_1-z_2}{2i}\right)>0$.

We deduce the following:

\begin{theorem}
In hyperbolic space, hyperbolic arc-length is a global and conformal parameter.
\end{theorem}
\end{example}

%
%
%
%
%
%
%
%


\section{Analytic curves in $\mathbb{R}^n$}\label{s:n}
\begin{definition}
Let $I\subset \mathbb{R}$ be a nonempty interval of arbitrary type which is not a singleton.
We say that $\gamma:I\to\mathbb{R}^n$ is analytic if  the following two conditions are satisfied:
\begin{enumerate} \item For every $t\in I$, the derivative $\gamma'$ exists in $\mathbb{R}^n$ and $\gamma'(t)\not=0$.
\item For every $t_0\in I$, there exists $\delta=\delta_{t_{0}}>0$ and a sequence $a_m\in \mathbb{R}^n$ such that $\sum_{m=0}^\infty (t-t_0)^ma_m$ converges in $(t_0-\delta, t_0+\delta)$ 
  and $\gamma(t)=\sum_{m=0}^\infty (t-t_0)^ma_m$ for any $t\in I\cap (t_0-\delta, t_0+\delta)$.

\end{enumerate}
\end{definition}

Let $\gamma:I \to \mathbb{R}^n$ be an analytic curve, with components $\gamma =
(\gamma_1, \gamma_2, ...,\gamma_n )$; that is, 
\[\gamma(t)=(\gamma_1 (t), \gamma_2
(t),...,\gamma_n (t) ).\]
 Then for any given point $t_0\in I$, we have $\gamma'_j(t_0)\not=0$ for some $1\leq j\leq n$. Therefore $\gamma_j$ is injective in a neighborhood of $t_0$. Consequently, $\gamma$ is locally injective.

We consider the complexification $\mathbb{C}^n$ of $\mathbb{R}^n$, taking
\[\mathbb{R}^n=(x_1,\ldots,x_n)\]
and 
\[\mathbb{C}^n=(x_1+iy_1,\ldots,x_n+iy_n).\]

Each $\gamma_j$ can be extended on a neighborhood of $[a,b]$ in $\mathbb{C}$ to a holomorphic mapping with values in $\mathbb{C}=\{x_j+iy_j\}$. Taking the intersections of these neighborhoods for $i=1,\ldots, n$, we see that $\gamma$ extends to a holomorphic mapping (which we also call $\gamma$) from a neighborhood $U$ of $[a,b]\subset \mathbb{C}$ to $\mathbb{C}^n$. 

By choosing a smaller neighborhood of $[a,b]$, we can assume that the extension $\gamma$ is locally injective. Indeed, since $\gamma_{\vert I}$ is locally injective, we may restrict our attention to an interval $I_0$ on which the curve $\gamma$ is injective. If the extension of $\gamma$ from $I_0$ to a small neighborhood of $I_0$ were not injective, then there would exist a nested sequence of open sets $U_k$ whose intersection is $I_0$, and points $z_k\not=w_k$ in each $U_k$ such that $\gamma(z_k)=\gamma(w_k)$. Without loss of generality, we can assume that $z_k\to z\in I_0$ and $w_k\to w\in I_0$. Therefore, $\gamma(z)=\gamma(w)$. Since $\gamma$ is injective on $I_0$, $z=w$.  But $\gamma'_j(z)\not=0$ for some $1\leq j\leq n$. Therefore the extension of $\gamma_j$ is injective on a neighborhood of $z$, which implies that $\gamma$ is injective on a neighborhood of $z$. This contradicts the fact that for every $k$, we have $\gamma(z_k)=\gamma(w_k)$.

Let $\gamma:[a,b]\to\mathbb{R}^n$ be an analytic curve in $\mathbb{R}^n$ and let $N$ be a natural number. Set  
\[ \Gamma(t)= \big(\gamma(t),\gamma'(t),\gamma''(t),\ldots,\gamma^{(N)}(t)
\big)\]
and 
\[J=\Gamma([a,b])= \{\big(\gamma(t),\gamma'(t),\gamma''(t),\ldots,\gamma^{(N)}(t)
\big): t\in [a,b]\}\subset \mathbb{R}^{nN}.\]
Let $\Omega\subset \mathbb{C}^{nN}$ be an open subset containing $J$, where we have the identification
\[\mathbb{R}^m=\{(\mathrm{Re} z_1, \ldots,\mathrm{Re} z_m); \ (z_1,\ldots,z_m)\in\mathbb{C}^m\}\]
with $m=nN$.
Let $F:\Omega\to \mathbb{C}$ be a holomorphic function in $\Omega$ such that $F(J)\subset (0,+\infty)$. 
Finally, let us set
\[S(t)=\int_a^t F(\Gamma(t))dt, \ a\leq t\leq b.\]
The function $t\mapsto S(t)$ is continuous and bijective from $[a,b]$ to $[A,B]$, for some $A,B$ in $\mathbb{R}$, and it defines a new parametrization of the curve $\gamma$, namely, $\delta=\gamma\circ S^{-1}$, where $S^{-1}:[A,B] \to [a,b]$ is the inverse of $S$.

We have the following:

\begin{theorem}\label{th:2}
The function  $\delta$ is also analytic.
\end{theorem}
The proof of Theorem \ref{th:2} is similar to that of Theorem \ref{th:1}. It can be done first in the case where $\gamma$ is injective and the result follows easily in the general case where $\gamma$ is locally injective

\begin{remark}  Since in this section $\gamma_j$ takes real values for t real, its holomorphic extension, denoted also by $\gamma_j,$ satisfies $\overline{\gamma_j(\overline z)} = \gamma_j(z);$ that is why in this section we do not need extra variables for the conjugates of the derivatives of $\gamma_j$ evaluated at $\overline z.$ 
\end{remark}

Using a similar technique we can obtain a theorem similar to Theorem 3.2, but for analytic curves in $\mathbb{C}^n.$


\section{Extendability with respect to spherical arc-length}\label{s:ext}

Consider an analytic curve $\gamma:(a,b)\to \mathbb{C}$ ($-\infty\leq a < b\leq +\infty$) under its maximal extension form with respect to arc-length.

If the set of limit points of $\gamma(t)$ when $t\to b$ has more than one element, then, obviously, any reparametrization we choose for $\gamma$ will not be extendable beyond $b$.

Suppose now that the set of limit points of $\gamma(t)$ when $t\to b$  is a singleton. If this point is in $\mathbb{C}$, then no matter what parametrization we choose for the curve, it will not be extendable beyond $b$, since $\gamma$ is endowed with the maximal parametrization by arc-length (Theorem 3.1 of \cite{NP2}).

The remaining case is when $\gamma(t)$ converges to $\infty$ as $t\to b$ (that is, for any compact subset $K$ of the complex plane, we have $\gamma(t)\cap K=\emptyset$ for all $t$ large enough). Thus, we have to consider the case where $b=+\infty$. In this case, using the arc-length parameter, the curve cannot be extended through $b$ since $b$ is not in $\mathbb{R}$ (and arc-length becomes infinite). But we can use instead the spherical arc-length, and it may happen that the spherical arc-length is finite. Thus, we consider the curve $\displaystyle \frac{1}{\gamma(t)}$ and we examine whether this curve is extendable using the spherical arc-length parametrization. This may happen indeed, and we give an example.

Consider a straight line, say $\mathbb{R}$, and let $\gamma:(-\infty,+\infty)\to \mathbb{C}$ be the inclusion map $\gamma(t)=t$. Then $\gamma$ is parametrized by arc-length. We have $\lim_{t\to +\infty}\gamma(t)=\infty$. 

By an easy computation, the spherical arc-length of the real axis is equal to $\pi$. 
Indeed, we have in the case considered $\vert \gamma'(t)\vert=1$, and the spherical length becomes
\[\int_{-\infty}^{+\infty} \frac{dt}{1+t^2}=\arctan \big]_{-\infty}^{+\infty} =\frac{\pi}{2}-(-\frac{\pi}{2})=\pi.\]

Thus, the parametrization of $\gamma$ by the arc length parameter on $(-\infty, +\infty)$  can be replaced by the spherical arc-length parametrization as a function defined on the interval $(-\frac{\pi}{2},\frac{\pi}{2})$.  The corresponding function is
\[\delta:(-\frac{\pi}{2},\frac{\pi}{2})\to \mathbb{R}\subset \mathbb{C}\]
given by $\delta(u)=\tan u$. 
The limit of the map $\displaystyle \frac{1}{\gamma}$ when $t\to+\infty$ is obviously $0$. 
We have $\lim_{u\to -\frac{\pi}{2}}\frac{1}{\delta(u)}=0$.
 
Thus, $\delta$ can be continued to a second real line: when the initial real line where $\delta$ is defined approaches $+\infty$, then the second real line will approach $-\infty$. The domain of definition using the spherical arc-length parametrization, instead of being $(-\frac{\pi}{2},\frac{\pi}{2})$ will become $(-\frac{\pi}{2},\frac{3\pi}{2})$. 

Continuing in this way to the right and to the left, we find that the domain of definition of $\delta$ is $(-\infty,+\infty)$. The image of this curve $\gamma$ is winding countably many times around the real axis, equipped with the spherical parmetrization by arc-length.

This gives an example of curves which are non-extendable using the Euclidean arc-length. The introduction of spherical arc-length allows to go over the point $\infty$. 

We note that the spherical arc-length parameter for the curves $\gamma$
and $\frac{1}{\gamma}$ is the same because the map $z\mapsto 1/z$  is an isometry for spherical arc-length.

Finally, we note that the curve $\delta:(-\infty,+\infty)\to \mathbb{C}\cup \{\infty\}$ is not extendable, no matter what parametrization we use, because the set of limit points at $+\infty$ is $\mathbb{R}\cup \{\infty\}$, which is not a singleton.

From the example of the real line that we gave, one can see that it may happen that an analytic curve $\gamma:(c,+\infty)\to\mathbb{C}$ ($c>0$) satisfying $\lim_{t\to+\infty}\gamma(t)=\infty$ extends analytically to $(c,+\infty)\cup\{\infty\}\cup (-\infty, c')$ $(c'<0$) where we identify the points $+\infty$, $\infty$ and $-\infty$. For this to hold, it suffices that $\frac{1}{\gamma(t)}$ extends analytically from $(0,1/c)$ to $(1/c',1/c)$ with $\frac{1}{\gamma{(\frac{1}{t})}}\big\vert_{t=0}=0$. We conclude that the point $\infty$ can be in the interior of the domain where a conformal parameter  for an analytic curve in $\mathbb{C}\cup\{\infty\}$ varies. 

Our previous definitions may be unified as follows:
\begin{definition}
A curve $\gamma$ in $\mathbb{C}\cup\{\infty\}$ is said to be analytic with respect to a parametrization with parameter $t$ if for every value $t_0\in\mathbb{R}\cup\{\infty\}$ in the range of the parameter $t$ there exists an open set $\Delta$ in the space where the parameter varies (which may be an interval, the union of two intervals with the point $\infty$ common to them, etc.) containing a point $t_0$, an open set $V$ in $\mathbb{C}\cup\{\infty\}$, such that $t_0\in\Delta\subset V$, and an injective meromorphic function $\phi:V\to \mathbb{C}\cup\{ \infty\}$ such that $\phi_{\vert\Delta}=\gamma_{\vert\Delta}$. 
\end{definition}

The example of the real line $\gamma(t)=t$, $t\in\mathbb{R}$, shows that for some parametrization of $\gamma$ the point $\infty=+\infty=-\infty$ can be in the interior of the set where the parameter $t$ can be extended and still give an analytic curve. However, we claim that this cannot happen if the parameter $t$ is the spherical arc-length. The reason is that spherical arc-length is a global conformal parameter for any analytic curve in $\mathbb{C}\cup\{\infty\}$. Indeed, if $\gamma$ can be extended through the point $\infty=+\infty=-\infty$, its spherical length is equal to that of $\frac{1}{\gamma(\frac{1}{t})}$ around $0$, which must be finite. This contradicts the fact that $t$ is the spherical arc-length and approaches infinity at an interior point where the parameter varies.

From the previous discussion, we deduce the following
\begin{theorem}
Spherical arc-length is a global conformal parameter for any analytic curve in $\mathbb{C}\cup\{\infty\}$. 
\end{theorem}
Finally, we mention that all the previous definitions are special cases of the definition of analytic curve on a Riemann surface, a subject we shall develop elsewhere.


\end{document}